\documentclass[11pt,reqno,a4paper]{amsart}
 \usepackage{amsfonts}
\usepackage{epsfig}
\usepackage{graphicx}
\usepackage{amsmath}
\usepackage{amssymb}
\usepackage{colordvi}
\usepackage{times,colordvi,amsmath,epsfig,float,multicol,subfigure}

\usepackage{amsfonts, amssymb, amsmath, amscd, amsthm, txfonts, graphicx, color, enumerate}

\usepackage{times,colordvi,float,multicol,subfigure}

\usepackage{ amscd, color, amsthm, txfonts, enumerate}

\newcommand{\red}[1]{\textcolor{red}{#1}}

 \allowdisplaybreaks

\numberwithin{equation}{section}
\numberwithin{figure}{section}
\theoremstyle{plain}
\newtheorem{theorem}{Theorem}[section]
\newtheorem{proposition}[theorem]{Proposition}

\newtheorem{corollary}[theorem]{Corollary}

\theoremstyle{definition}
\newtheorem{remark}[theorem]{Remark}
\newtheorem{example}[theorem]{Example}
\newtheorem{definition}[theorem]{Definition}

\numberwithin{equation}{section}

\title[Mean dimension and metric mean dimension]{Properties of mean dimension and metric mean dimension coming from the topological entropy}
\author{Jeovanny de Jesus Muentes  Acevedo \quad and \quad Carlos Rafael Payares Guevara}

\address{Jeovanny de J. Muentes Acevedo, Facultad de Ciencias B\'asicas,  Universidad Tecnol\'ogica de  Bolivar, Cartagena de Indias - Colombia}
\email{jmuentes@utb.edu.co}

\address{Carlos R. Payares Guevara, Facultad de Ciencias B\'asicas,  Universidad Tecnol\'ogica de  Bolivar, Cartagena de Indias - Colombia}
\email{cpayares@utb.edu.co}

\begin{document}

\begin{abstract}
In   the late 1990's, M. Gromov   introduced the notion of mean dimension for a continuous map, which is, as well as the topological entropy, an invariant under topological conjugacy.  The concept of metric mean dimension for a dynamical system was   introduced by Lindenstrauss and Weiss in 2000.   In  this paper we    will verify which properties coming from the topological entropy  map are valid for both mean dimension and metric mean dimension. In particular, we will prove that the metric mean dimension map is not continuous anywhere on the set consisting of continuous maps on both the Cantor set,   the interval or  the circle.  Finally we prove that the metric mean dimension on the set consisting of continuous map on the interval and on the circle is not lower semi-continuous.   \end{abstract}

\keywords{mean dimension, metric mean dimension, topological entropy, non-autonomous dynamical systems, box dimension}

\subjclass[2010]{	54H20, 	37E05, 	37A35}

\date{\today}
\maketitle

%\tableofcontents

%%%%%%%%%%%%%%%%%%%%%%%
\section{Introduction}

Let $X$ be a compact metric space with metric $d$. The notion of mean dimension for a topological dynamical system $(X,\phi)$, which will be denoted by $\text{mdim}(X,\phi)$, was introduced by M. Gromov in \cite{Gromov}. It is another  invariant under topological conjugacy.  Applications of the mean dimension to the embedding of dynamical systems problem  can be found in \cite{GTM}, \cite{Gutman}, \cite{lind2}, \cite{lind}, \cite{lind3}, \cite{lind4}  (these applications are summarized in the introduction of \cite{Fagner}).  

\medskip

Lindestrauss and Weiss in  \cite{lind},  introduced the notion of  metric mean dimension for any continuous map $\phi$ on $X$. It notion depends on the metric $d$ on $X$ (consequently it is not invariant under topological conjugacy) and it is zero for any map with finite topological entropy. Some well-known properties of the topological entropy are valid for both the mean dimension and the metric mean dimension. 
 Then it is a very interesting work to study which other  properties of the entropy are maintained  by the mean dimension and the metric mean dimension.  

\medskip

Let $\text{dim}(X)$ be the topological dimension of $X$, $\text{dim}_{H}(X,d)$ is the Hausdorff dimension of $X$ with respect to $d$, $\underline{\text{dim}_{B}}(X,d)$ is the lower box dimension of $X$ with respect to $d$ and $\overline{\text{dim}_{B}}(X,d)$ is the upper box dimension of $X$ with respect to $d$. We have $$\text{dim}(X)\leq \text{dim}_{H}(X,d)\leq   \underline{\text{dim}_{B}}(X,d)\leq  \overline{\text{dim}_{B}}(X,d),$$  (see \cite{Kawabata}, Section II, A).  In \cite{lind}, Proposition 3.1, is proved that  
$\text{mdim}(X^{\mathbb{Z}},\sigma)\leq \text{dim}(X)$, where $\sigma$ is the shift map on $X^{\mathbb{Z}}$.  In Section 2 we will present the definitions of  the lower  metric mean dimension  of $(X,\phi,d)$ (denoted by $ 
 \underline{\text{mdim}_M}(X,\phi,d)$)  and the upper metric mean dimension  of $(X,\phi,d)$  (denoted by $\overline{\text{mdim}_M}(X,\phi,d)$).  
In \cite{Fagner}, Theorem  4.4 and \cite{VV}, Theorem 5,  is proved that  for $\mathbb{K}=\mathbb{Z}$ or $\mathbb{N},$ we have    $  \overline{\text{mdim}_{M}}(X^{\mathbb{K}}, \sigma, \tilde{d}) =  \overline{\text{dim}_{B}} (X,d) $ and $  \underline{\text{mdim}_{M}}(X^{\mathbb{K}}, \sigma, \tilde{d})= \underline{\text{dim}_{B}} (X,d). $ This fact implies that  $ \overline{\text{mdim}_{M}}(X,\phi,d) \leq \overline{\text{dim}_{B}} (X,d) $  and $  \underline{\text{mdim}_{M}}(X,\phi,d) \leq \underline{\text{dim}_{B}} (X,d) $  for any continuous map  $\phi$ on $X$ (see \cite{Fagner}, Proposition 4.5 and \cite{VV}, Remark 4). Consequently, if $ \overline{\text{dim}_{B}} (X,d) $ is finite, then $\text{mdim}_{M}:C^{0}(X)\rightarrow \mathbb{R}$ is bounded. 

\medskip

 One of the most studied problem on the topological entropy is its continuity (see \cite{block}, \cite{Sheldon}, \cite{Yano}).     Note that if $\phi $ is any differentiable map, then   $h_{top}(\phi )<\infty$  and therefore $\phi$ has zero metric mean dimension (see Remark \ref{ueurur}). Then,  $\text{mdim}_{M}:C^{r}(X)\rightarrow \mathbb{R}$, $\phi\mapsto \text{mdim}_{M}(X,\phi,d)$,   is obviously continuous for $r\geq1$. Hence,  it remains to analyze the continuity of $\text{mdim}_{M}:C^{r}(X)\rightarrow \mathbb{R}$ only for $r = 0$. Now,     Yano in \cite{Yano} proved   that the topological entropy map is continuous on any continuous map, defined on a manifold, with infinite topological entropy. On the interval there exists continuous maps  with positive metric mean dimension   (see Example \ref{exfagner}).   We could expect that the metric mean dimension map to be continuous on maps with  maximum  metric mean dimension   (that is, maps on the interval with   metric mean dimension equal to 1). 
However, Example \ref{kegdjfjfn} proves the opposite. On the other hand, {Block, in \cite{block}, proved that $h_{top}$ is not continuous anywhere on the set consisting of continuous map on the Cantor set. Furthermore, he showed that if $X$ is the interval  or the circle, the topological entropy map is not continuous on continuous maps on $X$ with finite topological entropy. 
In Section 4 we will prove that    $\text{mdim}_{M}:C^{0}(X)\rightarrow \mathbb{R}$ is not continuous anywhere when  $X$ is the Cantor set, the interval or the circle. A corollary of the proof of  this  fact is that the set $\{\phi\in C^{0}([0,1]):  \text{mdim}_{M}([0,1],\phi,|\cdot |)=1\}$ is   dense  in $C^{0}([0,1])$ (see Corollary \ref{hdhdhfhf}). }
 
\medskip

 The map  $h_{top}: C^{0}([0,1])\rightarrow \mathbb{R}\cup \{\infty\}$ is lower semi-continuous (see \cite{Misiurewicz}).   However $h_{top}:\mathcal{C}([0,1])\rightarrow \mathbb{R}\cup\{\infty\}$ is not lower semi-continuous, where $\mathcal{C}([0,1])$ is the set consisting of non-autonomous dynamical systems on $[0,1]$ (see \cite{K-S}).    $h_{top}:\mathcal{C}([0,1])\rightarrow \mathbb{R}\cup\{\infty\}$ is  lower semi-continuous on any constant sequence $(\phi, \phi,\dots)\in \mathcal{C}(X)$. We will prove that the metric mean dimension on both 
 $ {C}([0,1])$ and  $\mathcal{C}([0,1])$ is not lower semi continuous (see Section 5). 
 
\medskip

Next, we present the structure of the paper.  In the next section we will present the definitions of the mean dimension and the metric mean dimension for continuous maps and non-autonomous dynamical systems. In Section 3 we will recall some well-known properties of the topological entropy map and we will discuss    what of these properties    are also valid for the    the mean dimension map and the metric mean dimension map. We will consider the cases of single maps and non-autonomous dynamical systems. % The proof of most of the results that we are going to present in Section 3 can be found in \cite{K-S}, \cite{lind}, \cite{JeoCTE} and \cite{Fagner}. 
Examples  \ref{ghndand} (which proves that the metric mean dimension depends on the metric on $X$) and \ref{mnbad} (which proves that the metric mean dimension of the direct product of two maps can be strictly smaller
that the sum of the metric mean dimension of each map) are provided by the authors.  The continuity and the semi continuity of $\text{mdim}_{M}$ will be discussed in Sections 4 and 5.

\section{Mean dimension and metric mean dimension}

Let $X$ be a compact metric space  endowed with a metric $d$. In this section we will define the mean dimension and the metric mean dimension for non-autonomous dynamical systems (see \cite{Fagner}) and for single continuous maps (see \cite{lind}), since  some properties will proved in the next section hold for both non-autonomous and  autonomous systems.  Suppose that  $\textit{\textbf{f}}=(f_n)_{n\in \mathbb N}$  is a  non-autonomous
dynamical system, where $f_n:X\to X$
is a continuous map for all $n\geq 1$. We write $(X,\textit{\textbf{f}},d)$ to denote a non-autonomous dynamical system \textit{\textbf{f}} on $X$. For $n,k\in \mathbb N$ define
$$
f_n^{(0)}:=I_X:= \text{the identity on }X\quad\quad \text{and} \quad \quad  f_n^{(k)}(x):=f_{n+k-1}\circ \dots \circ f_n(x)\quad\text{for }k\geq 1.
$$

 Set
\[ \mathcal{C}(X)=\{ (f_{n})_{n\in\mathbb{N}}: f_{n} \text{ is  continuous on }X\}\quad \text{and}\quad  {C}^{0}(X)=\{\phi:\phi \text{ is continuous  on }X\}.\]
Single continuous map are classified by topological conjugacy. Non-autonomous dynamical systems are classified by uniform equiconjugacy: 

\begin{definition}We say that  the systems  $\textit{\textbf{f}}=(f_{n})_{n\in\mathbb{N}}\in \mathcal{C}(X)$   and $\textit{\textbf{g}}=(g_{n})_{n\in\mathbb{N}}\in \mathcal{C}(Y)$, where $Y$ is a compact metric space, 
are \textit{uniformly equiconjugate} if there exists a equicontinuous sequence of  homeomorphisms $h_n: X\to Y$ (which will be called a \textit{uniform equiconjugacy})  so that
$h_{n+1}\circ f_n=g_n\circ h_n$, for all $n\in \mathbb N$. %  that is, the following diagram \[ \begin{CD}     X @>f_1>> X@>f_2>>\dots @>f_n>>X\\     @VVh_{1} V      @VV h_{2} V      @.     @VVh_{n+1}V\\    X @>g_1>> X@>g_2 >>\dots @>g_n>>X   \end{CD} \] is commutative for  all $n\in \mathbb N$. 
In the case where $h_n=h$, for all $n\in\mathbb N$, we say that
$ \textit{\textbf{f}}$ and $\textit{\textbf{g}}$ are \textit{uniformly conjugate}.\end{definition}

Given $\alpha $ an open cover of $X$ define
\[
\alpha_{0}^{n-1}=\alpha \vee f^{-1}_1(\alpha)\vee (f_1^{(2)})^{-1}(\alpha)\vee \dots \vee (f_1^{(n-1)})^{-1}(\alpha)
\]
and set 
\[\text{ord}(\alpha)=\sup_{x\in X}\sum_{U\in \alpha }1_{U}(x)-1\quad \quad \text{ and }\quad \quad \mathcal D(\alpha )
                    =\min_{\beta\prec\alpha}\text{ord}(\beta),\]
                   where $1_{U}$ is the indicator function and  $\beta \prec \alpha$ means that $\beta$ is a partition of $X$ finner than $\alpha$.
\begin{definition} The \textit{mean dimension} of $\textit{\textbf{f}}\in\mathcal{C}(X)$ is defined to be
\begin{equation}\label{def:meandimension}
\text{mdim}(X,\textit{\textbf{f}}\,)=\sup_{\alpha}\lim_{n\to\infty}\frac{\mathcal D(\alpha_{0}^{n-1})}{n}.
\end{equation}
The mean dimension of a continuous map $\phi:X\rightarrow X$ will be denoted by
$ 
\text{mdim}(X,\phi)$. \end{definition}

 \begin{remark}
   In \cite{Fagner}, Remark 2.2, we provide a list of properties of the mean dimension. 
 \end{remark}

    For any non-negative integer
$n$ we define $d_n:X\times X\to [0,\infty)$ by
$$
d_n(x,y)=\max\{d(x,y),d(f_1(x),f_1(y)),\dots,d(f_1^{(n-1)}(x),f_1^{(n-1)}(y))\}.
$$   Fix $\varepsilon>0$. We say that $A\subset X$ is a $(n,\textit{\textbf{f}},\varepsilon)$-separated set
if $d_n(x,y)>\varepsilon$, for any two  distinct points  $x,y\in A$. We denote by $\text{sep}(n,\textit{\textbf{f}},\varepsilon)$ the maximal cardinality of a $(n,\textit{\textbf{f}},\varepsilon)$-separated
subset of $X$. Furthermore, let $\text{cov}(n,\textit{\textbf{f}},\varepsilon)$ denotes the minimum number of $\varepsilon$-balls
in the $d_n$-metric to cover $X$. We say that $E\subset X$ is a $(n,\textit{\textbf{f}},\varepsilon)$-spanning set if
for any $x\in X$ there exists $y\in E$ so that $d_n(x,y)<\varepsilon$. Let $\text{span}(n,\textit{\textbf{f}},\varepsilon)$ be the minimum cardinality
of $(n,\textit{\textbf{f}},\varepsilon)$-spanning subset of $X$. Given an open cover $\alpha$, we say that $\alpha$ is a
$(n,\textit{\textbf{f}},\varepsilon)$-cover if the $d_n$-diameter of any element of $\alpha$ is less than $\varepsilon$. Let
$\text{cov}(n,\textit{\textbf{f}},\varepsilon)$ be the minimum cardinality of a $(n,\textit{\textbf{f}},\varepsilon)$-cover. By the compacity of $X$, $\text{sep}(n,\textit{\textbf{f}},\varepsilon)$, $\text{span}(n,\textit{\textbf{f}},\varepsilon)$  and $\text{cov}(n,\textit{\textbf{f}},\varepsilon)$ are finite real numbers.

 \begin{definition}
  The \emph{topological entropy} of $(X,\textit{\textbf{f}},d)$   is defined by      
  \begin{equation}\label{topent}h_{top}(\textit{\textbf{f}}\,)=\lim _{\varepsilon\to0} \text{sep}(\textit{\textbf{f}},\varepsilon)=\lim_{\varepsilon\to0}  \text{span}(\textit{\textbf{f}},\varepsilon)=\lim_{\varepsilon\to0}\text{cov}(\textit{\textbf{f}},\varepsilon),
\end{equation}
respectively, where $\text{sep}(\textit{\textbf{f}},\varepsilon)=\underset{n\to\infty}\lim \frac{1}{n}\log \text{sep}(n,\textit{\textbf{f}},\varepsilon)$, $\text{span}(\textit{\textbf{f}},\varepsilon)=\underset{n\to\infty}\lim \frac{1}{n}\log \text{span}(n,\textit{\textbf{f}},\varepsilon)$ and $\text{cov}(\textit{\textbf{f}},\varepsilon)=\underset{n\to\infty}\lim \frac{1}{n}\log \text{cov}(n,\textit{\textbf{f}},\varepsilon)$. \end{definition}

For any continuous map $\phi:X\rightarrow X$, we denote the topological entropy of $\phi$ by $h_{top}(\phi)$. 

 \begin{definition}
  We define the \emph{lower  metric mean dimension} of $(X,\textit{\textbf{f}},d)$  and the \emph{upper metric mean dimension} of $(X,\textit{\textbf{f}},d)$ by
  \begin{equation}\label{metric-mean}
 \underline{\text{mdim}_M}(X,\textit{\textbf{f}},d)=\liminf_{\varepsilon\to0} \frac{\text{sep}(\textit{\textbf{f}},\varepsilon)}{|\log \varepsilon|}\quad \text{ and }\quad\overline{\text{mdim}_M}(X,\textit{\textbf{f}},d)=\limsup_{\varepsilon\to0} \frac{\text{sep}(\textit{\textbf{f}},\varepsilon)}{|\log \varepsilon|},
\end{equation}
respectively. \end{definition}

It is not difficult to see that
$$
\underline{\text{mdim}_M}(X,\textit{\textbf{f}},d)=\liminf_{\varepsilon\to0} \frac{\text{span}(\textit{\textbf{f}},\varepsilon)}{|\log \varepsilon|}=\liminf_{\varepsilon\to0} \frac{\text{cov}(\textit{\textbf{f}},\varepsilon)}{|\log \varepsilon|}.
$$
The same holds for the upper metric mean dimension.

\medskip

For any continuous map $\phi:X\rightarrow X$, we denote the  lower  metric mean dimension and the  upper metric mean dimension of $\phi$ by  $  \underline{\text{mdim}_M}(X,\phi,d)$ and $ \overline{\text{mdim}_M}(X,\phi,d)$, respectively.

\medskip

The inequalities 
\[\text{mdim}(X,\textit{\textbf{f}}\,)\leq  \underline{\text{mdim}_M}(X,\textit{\textbf{f}},d)\leq \overline{\text{mdim}_M}(X,\textit{\textbf{f}},d) < h_{top}(\textit{\textbf{f}}\,)\] always hold. 
The proof for the first inequality for single map can be seen in \cite{lind}, Theorem 4.2, and for non-autonomous dynamical systems in \cite{Fagner}, Theorem 3.7. %Some   estimates of the topological entropy of non-autonomous dynamical systems can be found in  \cite{Zhang} and \cite{Zhu}.  

\begin{remark}\label{ueurur}
  Following the definition of topological entropy for non-autonomous dynamical systems introduced in \cite{K-S} one can see that
if the topological entropy of the system $(X,\textit{\textbf{f}},d)$ is finite then its metric mean dimension is zero.
\end{remark}

\begin{remark}
Throughout the paper, we will omit the underline and the overline  on the notations of the mean dimension and on the metric mean dimension when the result be valid for both cases. 
\end{remark}

\section{Properties coming from the topological entropy map}

In this section we will verify  what properties  coming from the topological entropy map (which are well-known)     are also valid for the map  $\Phi $ as being  the topological entropy map, the mean dimension map and the metric mean dimension map, for both, autonomous and non-autonomous dynamical systems.   These properties are: 
\begin{enumerate}[i)]
\item Invariance under conjugacy.
\item $\Phi(\phi^{p})=p\Phi (\phi)$, where $\phi$ is a continuous map or a sequence of continuous map on $X$.
\item Boundedness.
\item Monoticity for the case of non-autonomous dynamical systems. 
\item $\Phi(\phi\times \psi)= \Phi (\phi)+\Phi(\psi)$ for any continuous maps  $\phi$ and $\psi$ on $X$. 
\end{enumerate}
Throughout this section,    $\textit{\textbf{f}}=(f_{n})_{n\in\mathbb{N}}$ will be denotes a non-autonomous dynamical system in $ \mathcal{C}(X)$ and      $\phi: X\rightarrow X$ a continuous map. 

\medskip

Firstly, we will discuss about i).  One of the most important properties of the topological entropy is that it is invariant by topological conjugacy. It is also an   invariant under uniform equiconjugacy between  non-autonomous dynamical systems (see \cite{K-S} and \cite{JeoCTE}).    Mean dimension (for both, autonomous and non-autono\-mous dynamical systems) is invariant under topological conjugacy (see   \cite{Fagner}, Theorem 5.1).  In the next example we prove that the metric mean dimension for single continuous maps  depends on the metric $d$ on $X$. Consequently, it is not an invariant under topological  conjugacy between continuous maps and therefore it is not an invariant under uniform equiconjugacy between non-autonomous dynamical systems. 

\medskip

From now on, we will consider the metric $\tilde{d}$ on $X^{\mathbb{K}}$ defined by \begin{equation}\label{mnvc}\tilde{d}(\bar{x},\bar{y})= \sum_{i\in\mathbb{K}}\frac{1}{2^{|i|}}d(x_{i},y_{i}) \quad\text{ for }\bar{x}=(x_{i})_{i\in\mathbb{K}},\,  \bar{y}=(y_{i})_{i\in\mathbb{K}} \in X^{\mathbb{K}}.
\end{equation}

\begin{example}\label{ghndand}    Let $  \overline{\text{dim}_{B}} (X,d)$ and $  \underline{\text{dim}_{B}} (X,d)$ be respectively the upper and lower  box dimension  of $X$ with respect to $d$. In \cite{Fagner},  Theorem  4.4 and \cite{VV}, Theorem 5,  is proved that  for $\mathbb{K}=\mathbb{Z}$ or $\mathbb{N},$ we  have    \begin{equation}\label{dhhdhdf} \overline{\text{mdim}_{M}}(X^{\mathbb{K}}, \sigma, \tilde{d}) =  \overline{\text{dim}_{B}} (X,d) \quad \quad \text{and}\quad \quad \underline{\text{mdim}_{M}}(X^{\mathbb{K}}, \sigma, \tilde{d})= \underline{\text{dim}_{B}} (X,d). \end{equation}
Consequently, if $X_{1}=\mathbb{S}^{1}$ endowed with the metric $d_{1}$ inherited from $\mathbb{R}^{2}$, we have $$  \underline{\text{mdim}_{M}}(X_{1}^{\mathbb{K}}, \sigma, \tilde{d}_{1})= \overline{\text{mdim}_{M}}(X_{1}^{\mathbb{K}}, \sigma, \tilde{d}_{1})=1. $$
Let $X_{2}$ be the Koch curve endowed with the metric  $d_{2}$ inherited from the plane $\mathbb{R}^{2}$. Therefore,  ${\text{dim}_{B}} (X_{2},d_{2})=\frac{\log 4}{\log 3}$ and hence  
\begin{equation*} \overline{\text{mdim}_{M}}(X_{2}^{\mathbb{K}}, \sigma, \tilde{d}_{2}) = \underline{\text{mdim}_{M}}(X_{2}^{\mathbb{K}}, \sigma, \tilde{d}_{2})= {\text{dim}_{B}} (X_{2},d_{2})=\frac{\log 4}{\log 3}. \end{equation*}
%(see \red{Lectures on fractals and dimension theory,  Proposition 1.2.1.}) 
Since $X_{1}$ and $X_{2}$ are homeomorphic,  $(X_{1}^{\mathbb{Z}},\sigma)$  and $(X_{2}^{\mathbb{Z}},\sigma)$ are topologically conjugate, however their  metric mean dimensions are different.  
\end{example}

Set $ \mathcal{B}=\{\rho: \rho \text{ is a metric on }X\text{ equivalent to }d \}$  and take  
\[  \underline{\text{mdim}_M}(X, \textit{\textbf{f}}\,)=\inf_{\rho\in \mathcal{B}} \underline{\text{mdim}_M}(X,\textit{\textbf{f}},\rho)\quad \text{and}\quad  \overline{\text{mdim}_M}(X, \textit{\textbf{f}}\,)=\inf_{\rho\in \mathcal{B}} \overline{\text{mdim}_M}(X,\textit{\textbf{f}},\rho).\]
   
 Both  $ \underline{\text{mdim}_M}(X, \phi)$ and $ \underline{\text{mdim}_M}(X, \phi)$ are invariant under topological conjugacy for single continuous maps. For non-autonomous dynamical systems we have (see \cite{Fagner}, Theorem 5.1):

\begin{theorem}\label{edee344}  
  If $ \textit{\textbf{f}} $ and $ \textit{\textbf{g}} $ are uniformly equiconjugate by a sequence of homeomorphisms $(h_n)_{n\in\mathbb N}$ such that  $\inf_n \{d(h^{-1}_n(y_1),h_n^{-1}(y_2)), d(h_n(x_1),h_n(x_2))\}>0$  for any $x_1,x_2,y_{1},y_{2}\in X$, then
\begin{equation*}
\underline{\emph{mdim}_M}(X,\textit{\textbf{f}}\,)= \underline{\emph{mdim}_M}(X,\textit{\textbf{g}})\quad \text{and}\quad\overline{\emph{mdim}_M}(X,\textit{\textbf{f}}\,)= \overline{\emph{mdim}_M}(X,\textit{\textbf{g}}).
\end{equation*}
\end{theorem}

 About ii).    It is well-known that    $  
 h_{top}(\phi^{p})= p\, h_{top}(\phi)$ for any $p\geq1 .$  For the case of  non-autonomous dynamical systems, we have the inequality  $  
 h_{top}(\textit{\textbf{f}}^{\,(p)})\leq p\, h_{top}(\textit{\textbf{f}}\,), $  where   $\textit{\textbf{f}}^{\,(p)}=\{f_1^{(p)},f_p^{(p)},f_{2p}^{(p)},\dots\}$  (see \cite{K-S}, Lemma 4.2). In general, the equality $
 h_{top}(\textit{\textbf{f}}^{\,(p)})= p\, h_{top}(\textit{\textbf{f}}\,)$ is not valid, as we can see in the next example, which was given by Kolyada and Snoha in \cite{K-S}.

 \begin{example} Take $\psi: [0,1]\rightarrow [0,1]$ defined by $ \psi(x)=1-|2x-1|$ for any $x\in [0,1]$. Consider $\textbf{\textit{f}}=(f_{n})_{n\in\mathbb{N}}$, where  
 $$
f_n(x)= \begin{cases}
    \psi^{(n+1)/2}(x), &  \text{ if }n \text{ is odd}, \\
     x/2^{n/2}, & \hbox{ if }n\text{ is even},
      \end{cases}
$$ for any $n\in\mathbb{N}$. 
Then $h_{top}(\textbf{\textit{f}}^{(2)})= 0 $ and   $h_{top}(\textbf{\textit{f}}) \geq \frac{\log 2}{2}$.
 \end{example}
 
   The equality $
 h_{top}(\textit{\textbf{f}}^{\,(p)})= p\, h_{top}(\textit{\textbf{f}}\,)$ holds  if the sequence $\textit{\textbf{f}}=(f_{n})_{n\in\mathbb{N}}$ is equicontinuous (see \cite{K-S}, Lemma 4.4). 
 
 \medskip
 
 Next,  for any $p\geq 1$ we have  $ \text{mdim}(X,\phi^{p}) = p\, \text{mdim}(X,\phi)$  (see \cite{lind}, Proposition 2.7).  In \cite{Fagner}, Proposition 2.3, we proved that 
$   \text{mdim}(X,\textit{\textbf{f}}^{\,(p)})=p\, \text{mdim}(X,\textit{\textbf{f}}\,).
  $ For the metric mean dimension case, the inequality  $ \text{mdim}_{M}(X,\textit{\textbf{f}}^{\,(p)})\leq p\, \text{mdim}_{M}(X,\textit{\textbf{f}}\,)$ always hold. However, this inequality can be strict for single continuous map (and hence for non-autonomous dynamical systems), as we can see in the next example (see \cite{Fagner}, Example 4.7). 
  
 \begin{example}\label{exfagner} Take $g:[0,1]\rightarrow [0,1]$, defined by $x\mapsto |1-|3x-1||$, and  $0=a_{0}<a_{1}< \cdots <a_{n}<\cdots$, where $a_{n}=\sum_{k=1}^{n}6/\pi ^{2}k^{2}$ for $n\geq 1$. For each $n\geq 1$, let 
 $T_{n}: J_{n}:=[a_{n-1},a_{n}] \rightarrow [0,1] $ be the unique increasing affine map from $J_{n}$ (which has length $6/\pi ^{2}n^{2}$)  onto $[0,1]$ and take any strictly increasing sequence of natural numbers $m_{n}$.
 Consider the continuous map  $\psi:[0,1]\rightarrow [0,1]$ such that, for each $n\geq 1$, $\psi|_{J_{n}}= T_{n}^{-1}\circ g^{m_{n}}\circ T_{n}$.  In \cite{Fagner}, Example 4.7, we proved that  $ \text{mdim}_{M}([0,1],\psi,|\cdot|)=1$. On the other hand, by  Proposition 4.5 in \cite{Fagner}, $ \text{mdim}_{M}([0,1],\varphi,|\cdot|)\leq1$ for any continuous map  $\varphi: [0,1]\rightarrow [0,1]$. Therefore $ \text{mdim}_{M}([0,1],\psi^{2},|\cdot|)\leq1<2 \,\text{mdim}_{M}([0,1],\psi,|\cdot|)=2.$
    \end{example}
    
    Given that 0 and 1 are fixed points of $\psi$, it induces a map on the circle with metric mean dimension equal to 1.   
 
 \medskip
 
 A $s$-\textit{horseshoe}   for $\phi:[0,1]\rightarrow [0,1]$ is an interval $J\subseteq [0,1]$ which have a partition into $s$ subintervals $J_{1},\dots,J_{s}$ such that $J\subseteq \phi (\overline{J}_{i})$ for each $i=1,\dots, s$. If $\phi$ has a  $s$-horseshoe with $s\geq 2$, then $h_{top}(\phi)\geq \log s$.  In the above example, each  $J_{n}$ can be divided into $3^{m_{n}}$ closed intervals with the same length  $J_{n}(1),\dots,$ $  J_{n}({3^{m_{n}}})$, such that 
\[ \psi(J_{n}({i}))=J_{n} \quad\text{ for each } i\in\{1,\dots ,3^{m_{n}}\}\] (see \cite{Fagner}, Example 4.7).   Consequently, each $J_{n}$ is a $3^{m_{n}}$-horseshoe for $\psi$, for each  $n\in\mathbb{N}$. A question that arise of this fact is: 
how can we relate the metric mean dimension with the existence of horseshoes?

    \medskip

 iii) Boundedness. Note that if there exists a continuous map $\psi:X\rightarrow X$ with $h_{top}(\psi)>0$, then $h_{top}(\varphi)$ is unbounded as $\varphi$ ranges over the continuous maps on $X$ (since $h_{top}(\psi^{p})=ph_{top}(\psi)$ for any $p\geq 1$). Therefore, or $h_{top}(\phi)=0$ for any   $\phi$ or $h_{top}(\phi)$ is unbounded as $\psi$ ranges over the continuous maps on $X$.  
 
 \medskip

If the topological dimension of $X$ is finite, then  $\text{mdim}(X,\phi)=0$  (see \cite{lind}). On the other hand, if    $X$ has infinite topological dimension  and there exists a continuous map on $X$  with positive mean dimension  then the mean dimension is unbounded (since $\text{mdim}(X,\phi^{p})=p\,\text{mdim}(X,\phi)$ for any  $p\geq 1$), that is, or $\text{mdim}(X,\phi)=0$ for any   $\phi$ or $\text{mdim}(X,\phi)$ is unbounded as $\phi$ ranges over the continuous maps on $X$.   

 \medskip
 
The equalities in Equation \ref{dhhdhdf}  allow us to prove the following proposition (see \cite{Fagner}, Proposition 4.5 and \cite{VV}, Remark 4).

\begin{proposition}\label{erfdy}  For any   $\phi:X\rightarrow X$  we have 
$$\overline{\emph{mdim}_{M}}(X,\phi,d) \leq \overline{\emph{dim}_{B}} (X,d) \quad \text{ and }\quad \underline{\emph{mdim}_{M}}(X,\phi,d) \leq \underline{\emph{dim}_{B}} (X,d) .$$  
\end{proposition} 

For instances, if $X=[0,1]$, endowed with the metric $d(x,y)=|x-y|$ for $x,y\in X$, then 
$ 
\underline{\text{mdim}_{M}}(X,\phi,d) \leq \overline{\text{mdim}_{M}}(X,\phi,d) \leq 1$.  Example \ref{exfagner} shows that there exists a continuous map $\psi:X\rightarrow X$ such that $ 
\underline{\text{mdim}_{M}}(X,\psi,d) = \overline{\text{mdim}_{M}}(X,\psi,d) = 1$.

\medskip

On the other hand, $ 
\underline{\text{mdim}_{M}}(X,\textit{\textbf{f}},d)$ can be  unbounded as \textit{\textbf{f}}  ranges over the elements in $\mathcal{C}(X)$, regardless of the box dimension of $X$  (see    \cite{Fagner}, Example 4.8). Any  sequence of homeomorphisms on both the interval or   the circle has zero topological entropy (see \cite{K-S}, Theorem D).  % or  \cite{Zhu}, Theorem 2.1).    
Therefore, the metric mean dimension of any $\textbf{\textit{f}}$ on both the interval or   the circle   is zero.  In the next example we will see that there exist non-autonomous dynamical systems consisting of diffeomorphisms on a surface with infinite metric mean dimension.

\begin{example}\label{hfkenrkflr} Let $\phi:\mathbb{T}^{2}\rightarrow \mathbb{T}^{2}$ be the diffeomorphism induced by a hyperbolic matrix $A$ with eigenvalue $\lambda>1$, where $\mathbb{T}^{2}$ is the torus endowed with the metric $d$ inherited from the plane.    Consider $\textit{\textbf{f}}=(f_{i})_{i\in\mathbb{N}}$ where    $f_{i}=\phi^{2^i} $ for each $i\geq 1$.  We have $\text{Fix}(\phi^{n})=\lambda^{n}+\lambda^{-n}-2, $ where  $\text{Fix}(\psi)$ is the set consisting of fixed points of a continuous map $\psi$ (see \cite{Katok}, Proposition 1.8.1). Furthermore,  $$\text{sep}(n,  \textit{\textbf{f}},1/4)\geq \text{sep}(2^n,  \phi,1/4)\geq \text{Fix}(\phi^{2^n})=\lambda^{2^n}+\lambda^{-2^n}-2$$
(see \cite{Katok}, 
Chapter 3, Section 2.e). Therefore, 
$$\lim_{n\rightarrow \infty}\frac{\text{sep}(n,  \textit{\textbf{f}},1/4)}{n}\geq \lim_{n\rightarrow \infty}\frac{\log \lambda^{2^n}}{n}=\infty,$$
and hence $\text{mdim}_{M}(\mathbb{T}^{2},\textit{\textbf{f}}, d)=\infty$.
\end{example}

iv) {Monoticity} for the case of non-autonomous dynamical systems:  In  \cite{K-S}, Lemma 4.5, Kolyada and Snoha proved that $   h_{top}(\sigma^{i}(\textit{\textbf{f}}\,))\leq h_{top}(\sigma^{j}(\textit{\textbf{f}}\,))  $ for any $i\leq j,$  where $\sigma$ is the right shift   $\sigma ((f_n)_{n\in \mathbb N})=(f_{n+1})_{n\in \mathbb N}$. 
Furthermore, in \cite{JeoCTE},  Corollary 5.6, the author showed that if each $f_{i}$ is a homeomorphism then the equality holds, that is, the topological entropy for non-autonomous dynamical systems is independent on the   first maps on a sequence of homeomorphisms  $\textit{\textbf{f}}=(f_{n})_{n\in\mathbb{Z}}$. These properties are also valid for the mean dimension of non-autonomous dynamical systems  (see \cite{Fagner}, Proposition 2.4). Furthermore, we have if $\textit{\textbf{f}}=(f_n)_{n\in \mathbb N} $   converges uniformly to a continuous map $f:X\to X$, then   
\begin{align*}  \lim_{n\to\infty} h_{top}(\sigma^n(\textit{\textbf{f}}\,))&\leq  h_{top}(f) \text{\quad\quad (see  \cite{K-S}, Theorem E)}\\
\lim_{n\to\infty} \text{mdim}(X,\sigma^n(\textit{\textbf{f}}\,))&\leq  \text{mdim}(X,f) \text{\quad\quad (see  \cite{Fagner}, Theorem 2.6)}\\
\limsup_{n\to\infty} \text{mdim}_{M}(X,\sigma^n(\textit{\textbf{f}}\,))&\leq  \text{mdim}_{M}(X,f)\text{\quad\quad (see \cite{Fagner}, Theorem 3.8)}.
\end{align*}
All the inequalities can be strict (see \cite{K-S}, Section d, and \cite{Fagner}, Example 2.7). 

  \medskip
  
Finally, we discuss v).  Take $\phi:X\rightarrow X$ and $\psi:Y\rightarrow Y$ where $Y$ is a compact metric space with metric $d^{\prime}$. On $X\times Y$ we consider the metric \begin{equation}\label{bnm}(d\times d^{\prime})((x_{1},y_{1}),(x_{2},y_{2}))=d(x_{1},x_{2}) + d^{\prime}(y_{1},y_{2}),\quad \text{ for  }x_{1},x_{2}\in X \text{  and }y_{1},y_{2}\in Y.\end{equation}  The map  $\phi\times \psi : X\times Y\rightarrow X\times Y$ is defined to be $(\phi\times \psi)(x,y)=(\phi(x),\psi(y))$ for any $(x,y)\in X\times Y$. The equality $h_{top}(\phi\times \psi)=h_{top}(\phi)+h_{top}(\psi)$ always hold. Lindenstrauss in \cite{lind}, Proposition 2.8, proved that $\text{mdim}(X\times Y, \phi\times \psi)\leq \text{mdim}(X,\phi)+\text{mdim}(Y ,\psi)$ and this  inequality can be strict. For metric mean dimension we also have: \begin{proposition} The inequality  $$\emph{mdim}_{M}(X\times Y, \phi\times \psi,d\times d^{\prime})\leq \emph{mdim}_{M}(X,\phi,d)+\emph{mdim}_{M}(Y ,\psi,{d^{\prime}})$$ is valid. \end{proposition} 

In the next example we will prove that the above  inequality   can be strict. 

 \begin{example}\label{mnbad}  
 Let $(X,d)$ and $(Y,d^{\prime})$ be any compact metric spaces   such that  $$\text{dim}_{B}(X\times Y,d\times d^{\prime})<\text{dim}_{B}(X,d)+ \text{dim}_{B}(Y,d^{\prime}).$$ The metric $\tilde{d}\times \tilde{d}^{\prime}$ on $X^{\mathbb{Z}}\times Y^{\mathbb{Z}}$ is defined by $$(\tilde{d}\times \tilde{d}^{\prime})((\bar{x},\bar{y}),(\bar{z},\bar{w}))=  \sum_{i\in\mathbb{Z}}\frac{1}{2^{|i|}}d(x_{i},z_{i})+\sum_{i\in\mathbb{Z}}\frac{1}{2^{|i|}}d^{\prime}(y_{i},w_{i}),$$  for $\bar{x}=(x_{i})_{i\in\mathbb{Z}},\,  \bar{z}=(z_{i})_{i\in\mathbb{Z}} \in X^{\mathbb{Z}}, \bar{y}=(y_{i})_{i\in\mathbb{Z}},\,  \bar{w}=(w_{i})_{i\in\mathbb{Z}} \in Y^{\mathbb{Z}}
$ (see   \eqref{mnvc} and \eqref{bnm}). Furthermore, the metric  $(d\times d^{\prime})^{\ast}$ on  $(X\times Y)^{\mathbb{Z}}$ is given by \begin{align*}(d\times d^{\prime})^{\ast}((\overline{x,y}),(\overline{z,w}))&=  \sum_{i\in\mathbb{Z}}\frac{1}{2^{|i|}}(d\times d^{\prime})((x_{i},y_{i}),(z_{i},w_{i}))\\
&= \sum_{i\in\mathbb{Z}}\frac{1}{2^{|i|}}d(x_{i},z_{i})+\sum_{i\in\mathbb{Z}}\frac{1}{2^{|i|}}d^{\prime}(y_{i},w_{i}),\end{align*}  for $(\overline{x,y})=(x_{i},y_{i})_{i\in\mathbb{Z}} $ and $(\overline{z,w})=(z_{i},w_{i})_{i\in\mathbb{Z}} $ in $(X\times Y)^{\mathbb{Z}}$. Consequently, the bijection 
$$ \Theta :  (X\times Y)^{\mathbb{Z}}\rightarrow  X^{\mathbb{Z}}\times Y^{\mathbb{Z}}, \quad \text{given by }(x_{i},y_{i})_{i\in\mathbb{Z}} \mapsto ((x_{i})_{i\in\mathbb{Z}},(y_{i})_{i\in\mathbb{Z}}), $$ 
is an isometry and furthermore  the diagram  \[ \begin{CD}
     (X\times Y)^{\mathbb{Z}} @> \sigma >>  (X\times Y)^{\mathbb{Z}} \\
    @VV \Theta V      @VV \Theta V      \\
    X^{\mathbb{Z}}\times Y^{\mathbb{Z}} @> \sigma_{1}\times \sigma_{2} >>  X^{\mathbb{Z}}\times Y^{\mathbb{Z}}
  \end{CD}
\]
is commutative, where $\sigma$ is the shift on $(X\times Y)^{\mathbb{Z}} $, $\sigma_{1}$ is the shift on $X^{\mathbb{Z}}$ and  $\sigma_{2}$ is the shift on $Y^{\mathbb{Z}}$. It is clear that the metric mean dimension is invariant under isometric topological conjugacy. Therefore, 
\begin{align*}  \text{mdim}_{M}(X^{\mathbb{Z}}\times Y^{\mathbb{Z}}, \sigma_{1}\times \sigma_{2}, \tilde{d}_{1}\times \tilde{d}_{2})& =  \text{mdim}_{M}((X\times Y)^{\mathbb{Z}}, \sigma,(d\times d^{\prime})^{\ast})\\
&= \overline{\text{dim}_{B}} (X\times Y,d\times d^{\prime}) <\text{dim}_{B}(X,d)+ \text{dim}_{B}(Y,d^{\prime})\\
&= \text{mdim}_{M}(X^{\mathbb{Z}},\sigma_{1},\tilde{d}) + \text{mdim}_{M}(Y^{\mathbb{Z}}, \sigma_{2}, \tilde{d}^{\prime}).
\end{align*}
\end{example}
 
We finish this section with a remark.  If $\phi: X\rightarrow X$ is a factor of $\psi:Y\rightarrow Y$ then $h_{top}(\phi)\leq h_{top}(\psi)$. This fact is not valid for the (metric) mean dimension (see \cite{lind}). Indeed, let $\textbf{\textit{C}}$  be the Cantor  set and take a surjective continuous map $\Phi : \textbf{\textit{C}}\rightarrow [0,1]$.   We have that $([0,1]^{\mathbb{Z}},\sigma)$  is a factor of    $(\textbf{\textit{C}}^{\mathbb{Z}},\sigma)$ with factor mapping  \begin{align*}
    \tilde{\Phi} : \textbf{\textit{C}}^{\mathbb{Z}}&\rightarrow [0,1]^{\mathbb{Z}}\\
  (\dots, x_{-1},x_{0},x_{1},\dots)&\mapsto  (\dots, \Phi(x_{-1}),\Phi(x_{0}),\Phi(x_{1}),\dots).
\end{align*}   
Both  $([0,1]^{\mathbb{Z}},\sigma)$ and  $(\textbf{\textit{C}}^{\mathbb{Z}},\sigma)$ have infinite  topological entropy.
 However,  $$0=\text{mdim}(\textbf{\textit{C}}^{\mathbb{Z}},\sigma)<\text{mdim}([0,1]^{\mathbb{Z}},\sigma)=1$$
 and  $$\frac{\log 2}{\log 3}=\text{mdim}_{M}(\textbf{\textit{C}}^{\mathbb{Z}},\sigma,\tilde{d})<\text{mdim}_{M}([0,1]^{\mathbb{Z}},\sigma,\tilde{d})=1.$$

% \section{Maps with maximum mmdim on manifolds}
 
%Nesta secao queremos queremos verificar se os resultados feitos por Yano em \cite{Yano} para entropia topologica sao v\'alidos para a mmdim, no sentido de mostrar que o conjunto formado pelos homeos (aplicacoes continuas)    com mmdim maxima \'e generico.   
  
\section{On the continuity of the (metric) mean dimension} 
 
Block, in \cite{block}, studied the continuity of the topological entropy map on the set consisting of continuous maps on the Cantor set, the interval and the circle. Following the ideas of Block, in this section we will  study the continuity  of the mean dimension on the set of continuous maps on the product space $X^{\mathbb{K}}$, for $\mathbb{K}=\mathbb{Z}$ o $\mathbb{N}$, and furthermore the continuity of the  metric mean dimension map on the set consisting of continuous maps on      the product $X^{\mathbb{K}}$ for   $\mathbb{K}=\mathbb{Z}$ or $ \mathbb{N}$, the Cantor set,  %(which is a particular case of the case $X^{\mathbb{K}}$, since  it can be identified with  $ \{0,2\}^{\mathbb{N}}$), 
the  interval and the circle.  
%Note that if $X$ is the Cantor set, or the interval or the circle,  the metric mean dimension of any map on $C^{0}(X)$ is finite (it is less or equal to $\frac{\log 2}{\log 3}$ in the case of the Cantor set   and it is less or equal to  1 in the case of the interval and the circle by Proposition \ref{erfdy}). 

\medskip

On $C^{0}(X)$ we will consider the metric \begin{equation}\label{cbenfn} d(\phi,\varphi)=\max_{x\in X}d(\phi(x),\varphi(x))\quad \quad\text{ for any }\phi, \varphi \in   C^{0}(X).\end{equation}
Note that if $d^{\prime}$ is other metric on $X$ which induces the same topology that $d$ on $X$, then $ d^{\prime}(\phi,\varphi)=\max_{x\in X}d^{\prime}(\phi(x),\varphi(x))$,  for any $ \phi, \varphi \in   C^{0}(X)$, induces the same topology  that the metric $d$ defined on \eqref{cbenfn} on $C^{0}(X)$. Therefore, the continuity of  $\text{mdim}:C^{0}(X)\rightarrow \mathbb{R}\cup\{\infty\}$ and  $\text{mdim}_{M}:C^{0}(X)\rightarrow \mathbb{R}\cup\{\infty\}$ does  not depend on equivalent metrics on $X$.
 
 \medskip

Note that when the topological dimension of $X$ is finite,  for any continuous map $\phi$ on $X$ we have  $\text{mdim} (X, \phi)=0$ (see \cite{lind}), and therefore   $\text{mdim}:C^{0}(X)\rightarrow \mathbb{R}$, $\phi\mapsto \text{mdim}(X,\phi)$,  is the zero constant map. Hence, it remains to study the continuity of $\text{mdim}:C^{0}(X)\rightarrow \mathbb{R}$ when $\text{dim}(X)=\infty$.  Particular cases of infinite dimensional spaces are the product spaces $X^{\mathbb{Z}}$ or $X^{\mathbb{N}}$. 

 \begin{theorem}\label{bcbcbcb1} Take $\mathbb{K}=\mathbb{N}$ or $\mathbb{Z}$. 
If there exist $\psi\in C^{0}(X^{\mathbb{K}})$ with positive (metric) mean dimension, then the map $\emph{mdim}_{M}:C^{0}(X^{\mathbb{K}})\rightarrow \mathbb{R}\cup \{\infty\}$ (the map $\emph{mdim} :C^{0}(X^{\mathbb{K}})\rightarrow \mathbb{R}\cup \{\infty\}$) is not continuous anywhere. 
\end{theorem}
\begin{proof} We will prove the case $\mathbb{K}=\mathbb{N}$  (the case $\mathbb{K}=\mathbb{Z}$  is analogous). 
Fix a continuous map $\phi:X^{\mathbb{N}}\rightarrow X^{\mathbb{N}}$. Set $\bar{x}=(x_{1},x_{2},\dots)\in X^{\mathbb{N}}$ and assume that  $\phi(x_{1},x_{2},\dots)=(y_{1}(\bar{x}),y_{2}(\bar{x}),\dots)$. First, suppose that $\text{mdim}_{M}(X^{\mathbb{N}},\phi,\tilde{d})>0$ ($\text{mdim}(X^{\mathbb{N}},\phi)>0$). Consider the sequence of continuous maps on $X^{\mathbb{N}}$, $(\phi_{n})_{n\in\mathbb{N}}$, defined by 
$$   \phi_{n}(\bar{x})=(y_{1}(\bar{x}),y_{2}(\bar{x}),\dots,y_{n}(\bar{x}),x_{0},x_{0},\dots)\quad
\text{for any }n\in\mathbb{N}\text{ and some }x_{0}\in X.$$ 
It is clear that $\text{mdim}_{M}(X^{\mathbb{N}},\phi_{n},\tilde{d})=0$ ($\text{mdim} (X^{\mathbb{N}},\phi_{n} )=0$) for any $n\in\mathbb{N}$ and $\phi_{n}$ converges uniformly to $\phi$ as $n\rightarrow \infty$. 

Now, assume that $\text{mdim}_{M}(X^{\mathbb{N}},\phi,\tilde{d})=0$ ($\text{mdim} (X^{\mathbb{N}},\phi)=0$). Consider the sequence of continuous maps on $X^{\mathbb{N}}$, $(\phi_{n})_{n\in\mathbb{N}}$, where for any $n\geq 1$, $\phi_{n}$ is defined by 
$$   \phi_{n}(\bar{x})=(y_{1}(\bar{x}),y_{2}(\bar{x}),\dots,y_{n}(\bar{x}),z_{1}(\bar{x}),z_{2}(\bar{x}),\dots),\quad
\text{for any }\bar{x}\in X^{\mathbb{N}},$$ where $(z_{1}(\bar{x}),z_{2}(\bar{x}),\dots)=\psi(\bar{x})$ and $\psi$ is   continuous   with positive (metric) mean dimension. 
Then $\phi_{n}$ converges uniformly to $\phi$ as $n\rightarrow \infty$.  We can to prove that $\text{mdim}_{M}(X^{\mathbb{N}},\phi_{n},\tilde{d})=\text{mdim}_{M}(X^{\mathbb{N}},\psi,\tilde{d})$ for any $n\geq 1$ ($\text{mdim}(X^{\mathbb{N}},\phi_{n})=\text{mdim}(X^{\mathbb{N}},\psi)$ for any $n\geq 1$).
\end{proof}

Any $x \in [0, 1]$  is written in base 3 as
$$x=\sum _{i=n}^{\infty}x_{n}3^{-n}\text{ where }x_{n}\in \{0,1,2\}.$$  A number $x$ belongs to the middle third Cantor set  if and only if    no $x_n$  is equal to one.  Therefore, we can consider  \begin{equation}\label{nevsuidnf} \textit{\textbf{C}}=\{(x_{1},x_{2},\dots ): x_n=0,2\text{  for }n\in\mathbb{N}\}=\{0,2\}^{\mathbb{N}}\end{equation}
as being  the Cantor set endowed with the metric 
\begin{equation}\label{gsgdf} d((x_{1},x_{2},\dots),(y_{1},y_{2},\dots))=\sum_{i=1}^{\infty} 3^{-n}|x_{n}-y_{n}|=|\sum _{i=1}^{\infty}x_{n}3^{-n}-\sum _{i=1}^{\infty}y_{n}3^{-n}|. \end{equation}

  Bobok and Zindulka shown that  if  $X$ is an uncountable compact metrizable space of topological dimension zero, then given any $a \in  [0, \infty]$ there is a homeomorphism on $X$ whose topological entropy
is $a$. In particular, there exist homeomorphisms on the Cantor set with infinite topological entropy. In the next example we prove that there exists a homeomorphism on the Cantor ser with maximum metric mean dimension (the metric mean dimension of any map on $C^{0}(\textit{\textbf{C}} )$ is  less or equal to $\frac{\log 2}{\log 3}$ by Proposition \ref{erfdy}).

\begin{example}\label{gshfkf}
For any $k\geq 1$ and $z_{1},\dots z_{k}\in \{0,2\}$, set  $$\textit{\textbf{C}}_{z_{1},\dots,z_{k}}=\{(z_{1},\dots,z_{k},{x}_{1},x_{2},\dots): x_{j}\in \{0,2\}\text{ for }j\geq 1\}.$$ 
Note that if $(z_{1},\dots,z_{k})\neq (w_{1},\dots,w_{k})$, then $ \textit{\textbf{C}}_{z_{1},\dots,z_{k}}\cap \textit{\textbf{C}}_{w_{1},\dots,w_{k}}=\emptyset $ for $j,l=1,2$.  
Furthermore,   each  $\textit{\textbf{C}}_{z_{1},\dots,z_{k}}$ is homeomorphic to $\textit{\textbf{C}}$ via the  homeomorphism 
$$T_{z_{1},\dots,z_{k}}:\textit{\textbf{C}}_{z_{1},\dots,z_{k}}\rightarrow \textit{\textbf{C}},\quad (z_{1},\dots,z_{k},x_{1},x_{2},\dots)\mapsto (x_{1},x_{2},\dots),$$ which is Lipschitz.     
Take $\psi:\textit{\textbf{C}}\rightarrow \textit{\textbf{C}}$ defined by $$ \psi|_{\textit{\textbf{C}}_{z_{1},\dots,z_{k}}}= T_{z_{1},\dots,z_{k}}^{-1}\sigma^{k} T_{z_{1},\dots,z_{k}} .$$
For any $k\geq 1,$ take  $\varepsilon_k >0$ such  that $3^{-(k+1)}\leq\varepsilon_{k}<3^{-k}$.
 If $A\subseteq  \textit{\textbf{C}}$  is a  $ (nk,\sigma, \varepsilon_{nk})$-separated set, then  $$\tilde{A}=\{(z_{1},\dots,z_{k},x_{1},\dots ,x_{n},.\, .\, . ): (x_{1},\dots, x_{n},.\, . \, .) \in A\}\subseteq \textit{\textbf{C}}_{z_{1},\dots,z_{k}}$$ is a $(n,\psi ,\varepsilon_{k})$-separated set. 
 Therefore 
$  \text{sep}(n,\psi,\varepsilon_{k}) \geq 2^{nk}$ and hence 
\begin{align*}
 \frac{\log\text{sep}(n,\psi,\varepsilon_{k}) }{n|\log \varepsilon_{k}|}& \geq  \frac{\log(2^{nk}) }{n\log (3^{k})}\geq  \frac{n{k}\log2}{nk\log 3}=\frac{\log 2}{\log 3}.
\end{align*}
Therefore 
$$
\frac{\log 2}{\log 3}\geq \overline{\text{mdim}_M}({\textit{\textbf{C}}},\psi,d)\geq \underline{\text{mdim}_M}({\textit{\textbf{C}}},\psi,d)\geq  \frac{\log 2}{\log 3},
$$
that is, $   {\text{mdim}_M}({\textit{\textbf{C}}},\psi,d)=\frac{\log 2}{\log 3}.$
\end{example}

It follows from \eqref{nevsuidnf},  Theorem \ref{bcbcbcb1} and Example \ref{gshfkf} that: 

\begin{theorem}\label{bcbcbcb}
The map $\emph{mdim}_{M}:C^{0}(\textit{\textbf{C}})\rightarrow \mathbb{R}$ is not continuous anywhere. 
\end{theorem}
%\begin{proof} Fix a continuous map $\phi:\textit{\textbf{C}}\rightarrow \textit{\textbf{C}}$. Set $\bar{x}=(x_{1},x_{2},\dots)\in\textit{\textbf{C}}$ and assume that  $\phi(x_{1},x_{2},\dots)=(y_{1}(\bar{x}),y_{2}(\bar{x}),\dots)$. First, suppose that $\text{mdim}_{M}(\textit{\textbf{C}},\phi,{d})>0$. Consider the sequence of continuous maps on $\textit{\textbf{C}}$, $(\phi_{n})_{n\in\mathbb{N}}$, defined by  $$   \phi_{n}(\bar{x})=(y_{1}(\bar{x}),y_{2}(\bar{x}),\dots,y_{n}(\bar{x}),1,1,\dots)\quad \text{for any }n\in\mathbb{N}.$$  It is clear that $\text{mdim}_{M}(\textit{\textbf{C}},\phi_{n},{d})=0$ for any $n\in\mathbb{N}$ and $\phi_{n}$ converges uniformly to $\phi$ as $n\rightarrow \infty$. Now, assume that $\text{mdim}_{M}(\textit{\textbf{C}},\phi,{d})=0$. Consider the sequence of continuous maps on $\textit{\textbf{C}}$, $(\phi_{n})_{n\in\mathbb{N}}$, where for any $n\geq 1$, $\phi_{n}$ is defined by  $$   \phi_{n}(\bar{x})=(y_{1}(\bar{x}),y_{2}(\bar{x}),\dots,y_{n}(\bar{x}),z_{1}(\bar{x}),z_{2}(\bar{x}),\dots),\quad \text{for any }\bar{x}\in\textit{\textbf{C}},$$ where $(z_{1}(\bar{x}),z_{2}(\bar{x}),\dots)=\psi(\bar{x})$ and $\psi$ is the map of Example \ref{gshfkf}.  Then $\phi_{n}$ converges uniformly to $\phi$ as $n\rightarrow \infty$. We can to prove that $\text{mdim}_{M}(\textit{\textbf{C}},\phi_{n},d)=\text{mdim}_{M}(\textit{\textbf{C}},\psi,d)$ for any $n\geq 1$.\end{proof}

 It is well-known that any  perfect, compact, metrizable, zero-dimensional space  is homeomorphic to the middle third Cantor set (see \cite{Engelking}, 6.2.A(c)).  Hence, suppose that $X$ is a perfect, compact, metrizable, zero-dimensional space and let $\psi:X \rightarrow \textit{\textbf{C}}$ be an homeomorphism. Consider the metric on $X$ given by  $$d_{\psi}(x,y)=d(\psi(x),\psi(y)) \quad \text{ for }x,y\in X,$$ where $d$ is the metric given in \eqref{gsgdf}. 
It follows from Theorem \ref{bcbcbcb}   that:

\begin{corollary}\label{hd} The map
$\emph{mdim}_{M}:C^{0}(X,d_{\psi})\rightarrow \mathbb{R}$ is not continuous anywhere. Therefore, for any   perfect, compact, metric, zero-dimensional space $(X,d)$, the map $\emph{mdim}_{M}:C^{0}(X,d)\rightarrow \mathbb{R}$ is not continuous anywhere. 
\end{corollary}

Proposition \ref{erfdy}  says for any   $\phi:X\rightarrow X$  we have 
$\overline{\text{mdim}_{M}}(X,\phi,d) \leq \overline{\text{dim}_{B}} (X,d) $  and $  \underline{\text{mdim}_{M}}(X,\phi,d) \leq \underline{\text{dim}_{B}} (X,d) .$   However, there exist compact metric spaces $X$ with positive box dimension such that any continuous map on $X$ has zero metric mean dimension, as we will see in the next example.

\begin{example}  Let  $A = \{0\} \cup \{1/n: n\geq 1\}$ endowed with the metric $d(x,y)=|x-y|$ for $x,y\in A$.  In \cite{Kawabata}, Lemma 3.1, is proved that   $\underline{\text{dim}_{B}}( A ) = 1/2.$  However, for any    countable, compact metric space $X$  and   any  continuous map $\phi  :
X \rightarrow X $ we have $h_{top}(\phi ) = 0$ (see \cite{Bobok}, Proposition 5.1). Consequently, $\text{mdim}_{M}(X,\phi,{d} ) = 0$ for any continuous map  $\phi$ on $X$.  
\end{example}

Note that the set $A$ in the above example is compact, countable   zero topological dimensional. An example such that the space $X$ is   compact, uncountable and  zero topological dimensional and  such that any continuous map on $X$ has zero metric mean dimension can be seen in  \cite{Bobok}, Proposition 5.4.

\medskip

Next, we will consider the cases $X=[0,1]$ and  $X=\mathbb{S}^{1}$ (the unitary circle) endowed with the metric inherited from the line. In \cite{Yano}, Corollary 1.1, Yano   proved that $h_{top}: C^{0}(X)\rightarrow \mathbb{R}\cup \{\infty\}$ is continuous on any map with infinite topological entropy. We could expect that $\text{mdim}_{M}$ to be continuous in maps with   metric mean dimension equal to 1. 
However, the following example will show the opposite.  

\begin{example}\label{kegdjfjfn}
Consider the continuous map $\psi$ constructed on Example \ref{exfagner}, which has metric mean dimension equal to 1. 
For each $n\geq 1,$ take $m_{n}=n$ and
$$
\psi_n(x)= \begin{cases}
    \psi(x), &  \text{ if }x\in[0,a_{n+1}], \\
    x, & \hbox{ if }x\in [a_{n+1},1].
      \end{cases}
$$
Thus $\psi_n$ converges uniformly to $\psi$ as $n\rightarrow \infty$.    Note that $ \text{mdim}_M([0,1],\psi_{n},|\cdot |)= 0$ for any $n\geq1$. Therefore,  $\text{mdim}_M$ is not continuous on $\psi$. \end{example}

\begin{theorem}\label{ntmtmgm}
The map $\emph{mdim}_{M}:C^{0}([0,1])\rightarrow \mathbb{R}$ is not continuous anywhere. 
\end{theorem}
\begin{proof}
Let $\phi \in C^{0}([0,1])$. First, we will suppose that $\text{mdim}_{M}([0,1],\phi, |\cdot|)<1.$   
Choose  a fixed point  $x_{0}$ of $\phi$. First, suppose that $x_{0}<1$. Fix $\varepsilon >0$. There exists $\delta\in (0, \min\{\varepsilon/2,1-x_{0}\})$ such that if $|x-x_{0}|<\delta$ then $|\phi(x)-x_{0}|<\varepsilon/2$.  Let $T:[x_{0},x_{0}+\delta/2]\rightarrow [0,1]$ be the unique increasing affine map from $[x_{0},x_{0}+\delta/2]$ onto $[0,1]$.  Define $\varphi$ on $[x_{0},x_{0}+\delta/2]$ by $\varphi= T^{-1}\psi T $,  where $\psi$ is the continuous map of the Example \ref{exfagner}.  Note that $\varphi(x_{0})=x_{0}$ and $\varphi (x_{0}+\delta/2)=x_{0}+\delta/2$. Define $\varphi$ linearly on $[x_{0}+\delta/2,x_{0}+\delta]$ onto $[x_{0}+\delta/2, \phi(x_{0}+\delta)]$ (or onto $[ \phi(x_{0}+\delta),x_{0}+\delta/2]$ if $\phi$ is decreasing on $[x_{0}+\delta/2,x_{0}+\delta]$) and take $\varphi(x)=\phi(x)$ for any $x\in [0,1]\setminus [x_{0},x_{0}+\delta]$. If  $x\in   [x_{0},x_{0}+\delta]$, then  $$  |\phi(x)-\varphi(x)| \leq |\phi(x)-x_{0}|+|x_{0}-\varphi(x)|<\varepsilon .$$ Consequently, $d(\phi,\varphi)<\varepsilon.$  Furthermore, $$ \text{mdim}_{M}([0,1], \varphi,|\cdot|)=   \text{mdim}_{M}([0,1], \psi,|\cdot|) = 1  >\text{mdim}_{M}([0,1], \phi,|\cdot|), $$
which proves the theorem for $x_{0}\neq 1$. If $x_{0}=1$, we can make the construction above on the interval $[x_{0}-\delta/2,x_{0}]$. 

Now, if   $\text{mdim}_M([0,1],\phi,|\cdot |)= 1$, % there exists  a $C^{1}$-map $\varphi$ on $[0,1]$ such that $d(\phi,\varphi)<\varepsilon$, which proves the theorem since the metric mean dimension of any $C^{1}$-map is zero. % Therefore 
we have $h_{top}(\phi)=\infty$ and hence the set consisting of periodic points of $\phi$ is infinite. Take a periodic point $x_{1}$ of $\phi$, being an accumulation point of the set consisting of periodic points of $\phi$.   Let $\varepsilon >0$. Choose $\delta_{1}>0$ such that $d(\phi(x),x_{1})<\varepsilon/4$ for any $x\in I_{1}= [x_{1}-\delta_{1},x_{1}+\delta_{1}]$. Let $\tilde{x}_{1}\neq x_{1} $ other periodic point of $\phi$ in $[x_{1}-\delta_{1}/2,x_{1}+\delta_{1}/2]$. Let $\phi_{1}$ be the continuous map defined for $x$ in the interval $[\tilde{x}_{1},x_{1}]$ (or on $[x_{1},\tilde{x}_{1}]$ if $x_{1}<\tilde{x}_{1}$) as $\phi_{1}(x)=T^{-1}\psi T(x)$, where $T:[\tilde{x}_{1},x_{1}]\rightarrow [0,1]$ is the unique  affine map from $[\tilde{x}_{1},x_{1}]$ onto $[0,1]$ (or  $T:[x_{1},\tilde{x}_{1}]\rightarrow [0,1]$ is the unique  affine map from $[x_{1},\tilde{x}_{1}]$ onto $[0,1]$   if $x_{1}<\tilde{x}_{1}$), $\phi_{1}(x)=\phi(x)$ if $x\in [0,1]\setminus I_{1}$, and to extend it  continuously and linearly on $  I_{1}\setminus [\tilde{x}_{1},x_{1}]$ (or on $  I_{1}\setminus [x_{1},\tilde{x}_{1}]$ if $x_{1}<\tilde{x}_{1}$). Note that $d(\phi,\phi_{2})<\varepsilon/2$.  If $\text{mdim}_{M}([0,1],\phi|_{[0,1]\setminus I_{1}},d)=1$, as above we can to construct a perturbation $\phi_{2}$ of   $\phi_{1}$ on a neighborhood $I_{2}$ of other periodic  point which is an accumulation point of the set consisting of periodic points of $\phi$, such that $d(\phi_{1},\phi_{2})<\varepsilon/4$. Following this process, we can to construct a sequence of perturbations $\phi_{n}$ of $\phi_{n-1}$ on an interval $I_{n}$, such that $d(\phi_{n},\phi_{n-1})<\varepsilon/2^{n}$ until that $\text{mdim}_{M}([0,1],\phi|_{[0,1]\setminus \bigcup_{i=1}^{n}I_{i}},|\cdot |)=0$. Then $d(\phi_{n},\phi)<\varepsilon.$ In the same way as   in Example \ref{kegdjfjfn}, we can to construct a sequence of continuous maps $\psi_{k}$, converging to $\phi_{n}$ as $k\rightarrow \infty$,  with zero metric mean dimension on $\bigcup_{i=1}^{n}I_{i}$ and such that coincide with $\phi$ on $[0,1]\setminus \bigcup_{i=1}^{n}I_{i}$ for any $k\geq 1.$ Therefore  $\text{mdim}_{M}([0,1],\psi_{k},|\cdot |)=0$ for any $k\geq 1$ and $d(\psi_{k},\phi)<\varepsilon$ for a large enough $k$, which proves the theorem.
\end{proof}

A  consequence of Theorem \ref{ntmtmgm} is the following corollary.

\begin{corollary}\label{nbvbvvnbrr} 
The map $\emph{mdim}_{M}:C^{0}(\mathbb{S}^{1})\rightarrow \mathbb{R}$ is not continuous anywhere. 
\end{corollary}
\begin{proof}
Fix $\phi\in C^{0}(\mathbb{S}^{1})$ with metric mean dimension less than 1  and take $\varepsilon>0$. Since the set consisting of map with some periodic point is dense on $C^{0}(\mathbb{S}^{1})$, we can choose any continuous map $\varphi_{1}$ on $\mathbb{S}^{1}$ with a periodic point and such that $d(\phi,\varphi_{1})<\varepsilon/2$. We can modify the first argument of the proof of Theorem \ref{ntmtmgm}, with a periodic point of $\varphi_{1}$ replacing the role of the fixed point of $\phi$. Hence, we can construct a continuous map $\varphi$ on $\mathbb{S}^{1}$ with metric mean dimension equal to 1 and such that $d(\phi,\varphi)<\varepsilon,$ which proves the corollary.  

If the metric mean dimension of $\phi$ is equal to 1,   
we can argue as in the second part of the proof of the above theorem. % to show that for any $\varepsilon>0$ there exists a sequence on continuous maps on the circle with metric mean dimension equal to zero and  which converges to $\phi$. 
\end{proof}

 Now,     Yano in \cite{Yano} proved that the set consisting of continuous maps with infinite topological entropy defined  on any  manifold  is residual. In particular, the set consisting of continuous map with infinite topological entropy defined  on the interval and on the circle is residual.   A corollary of the proof of the above results is:
 
 \begin{corollary}\label{hdhdhfhf} If $X= [0,1]$ or $\mathbb{S}^{1}$, then 
 the set $\{\phi\in C^{0}(X):  \emph{mdim}_{M}(X,\phi,|\cdot |)=1\}$ is   dense  in $C^{0}(X)$.
 \end{corollary}

 \section{On the semi-continuity of the metric mean dimension}

 A real valued function $\varphi : X \rightarrow \mathbb{R}\cup \{\infty\}$ is called \textit{lower} (respectively \textit{upper})  \textit{semi-continuous on a point} $x\in X$ if $$\liminf_{y\rightarrow x}\varphi (y)\geq \varphi (x)\quad  (\text{repectively } \limsup_{y\rightarrow x}\varphi (y)\leq \varphi (x) ).    $$    $\varphi  $ is called \textit{lower} (respectively \textit{upper})  \textit{semi-continuous} if is  lower (respectively  {upper})  {semi-continuous on any point} of $ X$.

 \medskip
 
 Misiurewicz in \cite{Misiurewicz}, Corollary 1, proved that $h_{top}: C^{0}([0,1])\rightarrow \mathbb{R}\cup \{\infty\}$ is lower semi-continuous. For the case of the metric mean dimension we have:  
 
 \begin{proposition}\label{hfjdjehr} If $X=[0,1]$ or $\mathbb{S}^1$, then 
 $\emph{mdim}_{M}:C^{0}(X)\rightarrow \mathbb{R}$ is nor lower neither upper  semi-continuous. Furthermore,   
 $\emph{mdim}_{M}:C^{0}(X)\rightarrow \mathbb{R}$ is not  lower   semi-continuous on maps with metric mean dimension in $(0,1]$ and is not  upper   semi-continuous on maps with metric mean dimension in $[0,1)$. 
 \end{proposition}
 \begin{proof}
 In the second part of the proof of Theorem \ref{ntmtmgm}  we saw 
 that any continuous map on the interval with metric mean dimension equal to 1 can be approximated by a continuous map with zero metric mean dimension. Furthermore, we also proved that any continuous map with zero metric mean dimension can be approximated by a  continuous map with metric mean dimension equal to 1.  Using the same arguments of the proof of Theorem \ref{ntmtmgm} we can prove that any $\phi\in C^{0}([0,1])$ with metric mean dimension  in $(0,1)$ can be approximated by both a continuous map with metric mean dimension equal to 1 and a continuous map with metric mean dimension equal to 0.
 \end{proof}
 
  Next, Kolyada and Snoha in \cite{K-S}, Theorem F, showed that  $h_{top}:\mathcal{C}([0,1])\rightarrow \mathbb{R}\cup\{\infty\}$ is not lower semi-continuous, endowing  $\mathcal{C}([0,1])$   with the metric $$D((f_{i})_{i\in\mathbb{N}},(g_{i})_{i\in\mathbb{N}})=\sup_{i\in\mathbb{N}}\max_{x\in [0,1]}|f_{i}(x)-g_{i}(x)|.$$ Furthermore, they proved in Theorem G that $h_{top}:\mathcal{C}([0,1])\rightarrow \mathbb{R}\cup\{\infty\}$ is  lower semi-continuous on any constant sequence $(\phi, \phi,\dots)\in \mathcal{C}(X)$.
However, It follows from Proposition \ref{hfjdjehr} that:
 
 \begin{corollary} If $X=[0,1]$ or $\mathbb{S}^1$, then 
 $\emph{mdim}_{M}:\mathcal{C}(X)\rightarrow \mathbb{R}$ is nor lower neither upper  semi-continuous on any constant sequence $(\phi, \phi,\dots)\in \mathcal{C}(X)$. Consequently, $\emph{mdim}_{M}:\mathcal{C}(X)\rightarrow \mathbb{R}\cup\{\infty\}$ is nor lower neither upper  semi-continuous.  
 \end{corollary} 
 
From now on, we will consider   $X=[0,1]$ or $\mathbb{S}^1$. The next example proves that  there exist non-autonomous dynamical systems on $X$ with infinite metric mean dimension. Consequently $\text{mdim}_{M}:\mathcal{C}(X)\rightarrow \mathbb{R}\cup \{\infty\}$ is unbounded. 

 \begin{example}
 Take $\textbf{\textit{f}}=(f_{i})_{i\in\mathbb{N}}$ on $X$ defined by $f_{i} =\psi^{2^{i}}$ for each $i\in\mathbb{N}$, where $\psi$ is the map from Example \ref{exfagner}. It is not difficult to prove that $\text{mdim}_{M}(X,\textbf{\textit{f}},|\cdot|)=\infty$ (see Example \ref{hfkenrkflr}). 
 \end{example}
 
 We finish this work with the next result:
 \begin{theorem}\label{bnfuefnf}
$\emph{mdim}_{M}:\mathcal{C}(X)\rightarrow \mathbb{R}\cup\{\infty\}$ is not lower semi-continuous on any non-autonomous dynamical system  with non-zero metric mean dimension.  
\end{theorem}  
\begin{proof}
Let $\textit{\textbf{f}}=(f_{i})_{i\in\mathbb{N}}$ be a non-autonomous dynamical system with positive metric mean dimension. Let $\lambda_{n}$ be a sequence in $[0,1]$ such that $\lambda_{n}\rightarrow 1$ and $\lambda_{n}\cdots \lambda_{1}\rightarrow 0$ as $n\rightarrow \infty$. Take $\textbf{\textit{g}}_{n}=(\lambda_{n+i}f_{i})_{i\in\mathbb{N}}$. Thus $\textbf{\textit{g}}_{n}\rightarrow \textbf{\textit{f}}$ as $n\rightarrow \infty$. However, for any $x\in X$, $({g}_{n})^{(k)}(x)\rightarrow 0$ as $k\rightarrow \infty$. Consequently, the metric mean dimension of $\textbf{\textit{g}}_{n}$ is zero for each $n\in \mathbb{N}$. 
\end{proof}


\begin{thebibliography}{2}
\bibitem{block} Block, Louis. ``Noncontinuity of topological entropy of maps of the Cantor set and of the interval.'' \textit{Proceedings of the American Mathematical Society} 50.1 (1975): 388-393.
\bibitem{Bobok} Bobok, Jozef, and Ondrej Zindulka. ``Topological entropy on zero-dimensional spaces." \textit{Fundamenta Mathematicae} 162.3 (1999): 233-249.
\bibitem{Gromov} Gromov, Misha. ``Topological invariants of dynamical systems and spaces of holomorphic maps: I." \textit{Mathematical Physics, Analysis and Geometry} 2.4 (1999): 323-415.
\bibitem{GTM} Gutman, Yonatan. ``Embedding topological dynamical systems with periodic points in cubical shifts." \textit{Ergodic Theory and Dynamical Systems} 37.2 (2017): 512-538.
\bibitem{Gutman} Gutman, Yonatan, and Masaki Tsukamoto. ``Embedding minimal dynamical systems into Hilbert cubes." \textit{arXiv preprint arXiv:1511.01802} (2015).
\bibitem{Engelking} Engelking, Ryszard. \textit{General Topology}, Heldermann, Berlin, 1979.
\bibitem{Katok} Katok, Anatole, and Boris Hasselblatt. \textit{Introduction to the modern theory of dynamical systems}. Vol. 54. Cambridge university press, 1995.
\bibitem{Kawabata} Kawabata, Tsutomu, and Amir Dembo. ``The rate-distortion dimension of sets and measures." \textit{IEEE transactions on information theory} 40.5 (1994): 1564-1572.
\bibitem{K-S} Kolyada, Sergii, and Lubomir Snoha. ``Topological entropy of nonautonomous dynamical systems." \textit{Random and computational dynamics} 4.2 (1996): 205.
\bibitem{lind2} Lindenstrauss, Elon. ``Mean dimension, small entropy factors and an embedding theorem." \textit{Publications Mathématiques de l'Institut des Hautes Études Scientifiques} 89.1 (1999): 227-262.
\bibitem{lind} Lindenstrauss, Elon, and Benjamin Weiss. ``Mean topological dimension." \textit{Israel Journal of Mathematics} 115.1 (2000): 1-24.
\bibitem{lind3} Lindenstrauss, Elon, and Masaki Tsukamoto. ``From rate distortion theory to metric mean dimension: variational principle.'' \textit{IEEE Transactions on Information Theory} 64.5 (2018): 3590-3609.
\bibitem{lind4}  Lindenstrauss, Elon and   Masaki Tsukamoto.  ``Mean dimension and an embedding problem: an example.'' \textit{Israel
J. Math.} 199.573-584 (2014): 5-2.
\bibitem{Misiurewicz} Misiurewicz, Michal. ``Horseshoes for continuous mappings of an interval." \textit{Dynamical systems}. Springer, Berlin, Heidelberg, 2010. 125-135.
\bibitem{JeoCTE}  Muentes, Jeovanny. ``On the continuity of the topological entropy of
non-autonomous dynamical systems.'' \textit{Bulletin of the Brazilian Mathematical Society, New Series} 49.1 (2018): 89-106.
\bibitem{Sheldon} Newhouse, Sheldon E. ``Continuity properties of entropy." \textit{Annals of Mathematics} 129.1 (1989): 215-235.
\bibitem{Fagner} Rodrigues, Fagner Bernardini, and Jeovanny de Jesus Muentes Acevedo. ``Mean dimension and metric mean dimension for non-autonomous dynamical systems.'' \textit{arXiv preprint arXiv:1905.05367} (2019). 
%\bibitem{Seidler} Seidler, Gerald T. ``The topological entropy of homeomorphisms on one-dimensional continua." \textit{Proceedings of the American Mathematical Society} 108.4 (1990): 1025-1030.
\bibitem{VV} Velozo, Anibal, and Renato Velozo. ``Rate distortion theory, metric mean dimension and measure theoretic entropy." \textit{arXiv preprint arXiv:1707.05762} (2017).
\bibitem{Yano} Yano, Koichi. ``A remark on the topological entropy of homeomorphisms." \textit{Inventiones mathematicae} 59.3 (1980): 215-220.
%\bibitem{Zhang} Zhang, Jin-lian, and Lan-xin Chen. ``Lower bounds of the topological entropy for nonautonomous dynamical systems." \textit{Applied Mathematics-A Journal of Chinese Universities} 24.1 (2009): 76-82.
%\bibitem{Zhu} Zhu, Yuhun, Jinlian Zhang, and Lianfa He. ``Topological entropy of a sequence of monotone maps on circles." \textit{Journal of the Korean Mathematical Society} 43.2 (2006): 373-382.
%\bibitem{Zhu} Zhu, Yujun, et al. ``Entropy of nonautonomous dynamical systems." \textit{Journal of the Korean Mathematical Society} 49.1 (2012): 165-185.
\end{thebibliography}
\end{document}